\documentclass[12pt,reqno]{amsart}
\usepackage[left=80pt,right=80pt]{geometry}
\usepackage[usenames]{color}
\usepackage{amsmath}
\usepackage{amssymb}
\usepackage{amsthm}
\usepackage{enumitem}
\usepackage{float}
\usepackage{graphicx}
\usepackage{tikz}
\usepackage{subcaption}
\usepackage{cases}
\usepackage{blkarray}
\usepackage{tikz-qtree}
\usepackage{pgfplots}
\pgfplotsset{compat=1.17}

\usetikzlibrary{graphs,graphs.standard,calc, intersections, pgfplots.fillbetween}
\pgfdeclarelayer{ft}
\pgfdeclarelayer{bg}
\pgfsetlayers{bg,main,ft}

\usepackage{hyperref,url}
\hypersetup{
	colorlinks=true,
  linkcolor=black,          
  citecolor=black,         
  filecolor=black,      
  urlcolor=black           
}
\newtheorem{prop}{Proposition}
\newtheorem{lemma}[prop]{Lemma}
\newtheorem{theorem}[prop]{Theorem}

\newtheorem{corollary}[prop]{Corollary}
\theoremstyle{definition}
\newtheorem{definition}[prop]{Definition}
\newtheorem{remark}[prop]{Remark}
\newtheorem{example}[prop]{Example}
\newtheorem{conjecture}[prop]{Conjecture}

\newcommand{\N}{\mathbb{N}}
\newcommand{\R}{\mathbb{R}}

\newcommand{\seqnum}[1]{\href{https://oeis.org/#1}{\rm \underline{#1}}}
\allowdisplaybreaks
\newcommand{\mylabel}[2]{#2\def\@currentlabel{#2}\label{#1}}

\newcommand{\des}{\textnormal{des}}
\newcommand{\cycdes}{\textnormal{cycdes}}
\newcommand{\lev}{\textnormal{lev}}
\newcommand{\cyclev}{\textnormal{cyclev}}
\begin{document}
\tikzset{mystyle/.style={matrix of nodes,
        nodes in empty cells,
        row 1/.style={nodes={draw=none}},
        row sep=-\pgflinewidth,
        column sep=-\pgflinewidth,
        nodes={draw,minimum width=1cm,minimum height=1cm,anchor=center}}}
\tikzset{mystyleb/.style={matrix of nodes,
        nodes in empty cells,
        row sep=-\pgflinewidth,
        column sep=-\pgflinewidth,
        nodes={draw,minimum width=1cm,minimum height=1cm,anchor=center}}}

\title{The total number of descents and levels in tensor words and cyclic tensor words}

\author[SELA FRIED]{Sela Fried$^{\dagger}$}
\thanks{$^{\dagger}$ Department of Computer Science, Israel Academic College,
52275 Ramat Gan, Israel.
\\
\href{mailto:friedsela@gmail.com}{\tt friedsela@gmail.com}}
\author[TOUFIK MANSOUR]{Toufik Mansour$^{\sharp}$}
\thanks{$^{\sharp}$ Department of Mathematics, University of Haifa, 3103301 Haifa,
Israel.\\
\href{mailto:tmansour@univ.haifa.ac.il}{\tt tmansour@univ.haifa.ac.il}}

\maketitle

\begin{abstract}
We obtain an explicit formula for the total number of descents and levels in tensor words and cyclic tensor words of arbitrary dimension. We also establish the exact maximal number of descents in words and use it to obtain an upper bound on the maximal number of descents in tensor words.
\bigskip

\smallskip

\noindent
\textbf{Math.~Subj.~Class.:} 05A05, 05A15, 15A69.
\end{abstract}

\section{Introduction}

Let $k$ and $n$ be two natural numbers and let $[k]=\{1,2,\ldots,k\}$. Recall that if $w=w_1\cdots w_n\in[k]^n$ is a word over the alphabet $[k]$, then $1\leq i\leq n-1$ is a descent (resp.\ level) of $w$ if $w_{i}>w_{i+1}$ (resp.\ $w_{i}=w_{i+1}$). If the word is cyclic, then also $i=n$ is a descent (resp.\ level) of $w$ if $w_n > w_1$ (resp.\ $w_n=w_1$). One possibility for a generalization, that we pursue in this work, is to consider words in higher dimensions, thus regarding (standard) words as one-dimensional words. For example, if $m$ is an additional natural number, then two-dimensional words over $[k]$ are matrices $w=(w_{i,j})_{1\leq i\leq m,1\leq j\leq n}\in[k]^{mn}$. Here, a descent (resp.\ level) of $w$ is a pair of double indices $((i_1,j_1), (i_2,j_2))$ such that $w_{i_1,j_1}>w_{i_2,j_2}$ (resp.\ $w_{i_1,j_1}=w_{i_2,j_2}$) and either $(i_2,j_2)=(i_1+1,j_1)$ or $(i_2,j_2)=(i_1,j_1+1)$.

The purpose of this work is to establish a formula for the total number of descents (resp.\ levels) in multidimensional words of an arbitrary (but fixed) dimension. We refer to such words as tensor words, as formalized in Definition \ref{def;1}.

Descents and levels in words were studied, for example, in \cite{B, Kit}. Especially related to our work is the work of Mansour and Shattuck \cite{MS}, which studied common occurrences of patterns in matrix (i.e., two-dimensional) words. Nevertheless, they have not considered vertical descents (resp.\ levels), nor have they considered cyclic words. In contrast, Knopfmacher et al.\ \cite{K} considered cyclic words, but in a different context, namely that of staircase words.


Before we begin, let us introduce some notation. We denote by $\N$ (resp.\ $\R$) the set of natural (resp.\ real) numbers. For $n\in\N$, let $[n] = \{1,2,\ldots,n\}$ and, for $d\in\N$ and $\ell\in [d]$, we denote by $e_\ell$ the $\ell$th vector in the standard basis of $\R^d$. If $p$ is a condition, then $1_p$ equals $1$ if $p$ holds and $0$ otherwise. Fix $k,d\in\N$ and $m^{(d)}=(m_1,\ldots,m_d)\in\N^d$, to be used throughout this work.

\begin{definition}\label{def;1}
An \emph{$m^{(d)}$-tensor word over $[k]$} is a function $w\colon [m_1]\times\cdots\times [m_d]\to[k]$. The set of all $m^{(d)}$-tensor words over $[k]$ is denoted by $T(m^{(d)},k)$. A \emph{descent} (resp.\  \emph{level}) of an $m^{(d)}$-tensor word $w$ is a pair $(i,j)\in([m_1]\times\cdots\times [m_d])^2$, such that $w(i)>w(j)$ (resp.\ $w(i)=w(j)$) and $j=i+e_\ell$, for some $\ell\in[d]$. We denote by $\des(w)$ (resp.\ $\lev(w)$) the number of descents (resp.\ levels) of $w$. If the word $w$ is regarded as cyclic, then additional \emph{cyclic descents} (resp.\ \emph{cyclic levels}) are allowed, namely all $(i,j)\in([m_1]\times\cdots\times [m_d])^2$, such that there exists a unique $\ell\in[d]$ with $i_\ell =n, j_\ell=1$, and
$w(i)>w(j)$ (resp.\ $w(i)=w(j)$). Finally, we denote by $\cycdes(w)$ (resp.\ $\cyclev(w)$) the number of descents and cyclic descents (resp.\ levels and cyclic levels) of $w$.
\end{definition}

\begin{example}
Let
$$w=\begin{pmatrix}
3&1&2&1\\
2&2&2&2\\
1&2&3&3\\
\end{pmatrix}.$$
Then $w$ is a $(3,4)$-tensor word over $[3]$ (or every $k\geq 3$). The set of descents of $w$ is given by $\{((1,1),(1,2)), ((1,3),(1,4)), ((1,1),(2,1)),((2,1),(3,1))\}$ and the set of levels of $w$ is given by $$\{((2,1),(2,2)), ((2,2),(2,3)), ((2,3),(2,4)), ((3,3),(3,4)), ((2,2),(3,2)),((1,3),(2,3))\}.$$ Thus, $\des(w)=4$ and $\lev(w)=6$. The set of cyclic descents of $w$ is given by $$((3,4), (3,1)), ((3,2), (1,2)), ((3,3), (1,3)), ((3,4), (1,4))\}$$ and the set of cyclic levels of $w$ is given by $((2,4), (2,1))\}$. Thus, $\cycdes(w)=8$ and $\cyclev(w)=7$.
\end{example}

\section{Main results}

Our main results are as follows: Let $a_d$ (resp.\ $b_d$) denote the total number of descents (resp.\ levels) of all $m^{(d)}$-tensor words over $[k]$. Then $$a_d=\frac{1}{2}\left(d\prod_{\ell\in[d]}m_{\ell}-\sum_{i\in[d]}\prod_{\ell\in[d],\ell\neq i}m_{\ell}\right)(k-1)k^{\prod_{\ell\in[d]}m_{\ell}-1}$$ and $b_d=2a_d/(k-1)$, i.e., $$b_d=\left(d\prod_{\ell\in[d]}m_{\ell}-\sum_{i\in[d]}\prod_{\ell\in[d],\ell\neq i}m_{\ell}\right)k^{\prod_{\ell\in[d]}m_{\ell}-1}.$$
Similarly, let $f_d$ (resp.\ $g_d$) denote the total number of descents and cyclic descents (resp.\ levels and cyclic levels) of all $m^{(d)}$-tensor words over $[k]$. Then $$f_d=\frac{1}{2}\left(d\prod_{\ell\in[d]}m_{\ell}\right)(k-1)k^{\prod_{\ell\in[d]}m_{\ell}-1}$$ and $g_d=2f_d/(k-1)$, i.e., $$g_d=\left(d\prod_{\ell\in[d]}m_{\ell}\right)k^{\prod_{\ell\in[d]}m_{\ell}-1}.$$

In the last section we provide an upper bound on the maximal number of descents and cyclic descents, namely, we prove that, for every $w\in T(m^{(d)},k)$, we have
\[\cycdes(w)\leq\sum_{i\in[d]}\left(\prod_{\ell\in[d],\ell\neq i}m_{\ell}\right)\left\lfloor \frac{m_{i}(k-1)}{k}\right\rfloor.
\]

\subsection{Noncyclic words}
Let $m_{d+1}\in\N$ and let $\ell$ be a nonnegative integer. Set $m^{(d+1)}=(m_1,\ldots,m_{d+1})$. We denote by $D_{m^{(d)},m_{d+1},\ell}$ the number of $m^{(d+1)}$-tensor words $w$ such that $\des(w)=\ell$. For $w\in T(m^{(d)},k)$, we denote by $D_{m^{(d)},m_{d+1},\ell}(w)$ the number of $m^{(d+1)}$-tensor words $w'$, such that $\des(w')=\ell$ and such that $w'(i,1)=w(i)$, for every $i\in [m_1]\times\cdots\times [m_d]$. Let $F_{d+1}(t, x)=\sum_{m_{d+1}\geq 1}\sum_{\ell\geq 0} D_{m^{(d)},m_{d+1},\ell}t^\ell x^{m_{d+1}}$ be the generating function of the $D_{m^{(d)},m_{d+1},\ell}$s, and, for $w\in T(m^{(d)},k)$, we denote by $F_{d+1}(t,x,w)=\sum_{m_{d+1}\geq 1}\sum_{\ell\geq 0}D_{m^{(d)},m_{d+1},\ell}(w)t^\ell x^{m_{d+1}}$ the generating function of the $D_{m^{(d)},m_{d+1},\ell}(w)$s. Notice that $$F_{d+1}(1,x,w)=\frac{x}{1-k^{\prod_{\ell\in[d]} m_\ell}x}.$$

\begin{lemma}\label{lem;j1}
For $w\in T(m^{(d)},k)$, we have
\begin{align}
F_{d+1}(t,x,w)=\frac{xt^{\des(w)}}{1-xt^{\des(w)}}\left(1+\sum_{w'\in T(m^{(d)},k), w'\neq w}t^{\sum_{i\in[m_1]\times\cdots\times [m_d] } 1_{w(i)>w'(i)}}F_{d+1}(t,x,w')\right).\nonumber
\end{align}
\end{lemma}

\begin{proof}
It suffices to show that the generating function $F_{d+1}(t,x,w)$ satisfies the equation
\begin{equation}\label{k1}
F_{d+1}(t,x,w)=t^{\des(w)}x+x\sum_{w'\in T(m^{(d)},k)}t^{\des(w)+\sum_{i\in[m_1]\times\cdots\times [m_d] } 1_{w(i)>w'(i)}}F_{d+1}(t,x,w').
\end{equation} Indeed, let $v\in T(m^{(d+1)},k)$ such that $v(i,1)=w(i)$, for every $i\in [m_1]\times\cdots\times [m_d]$. If $m_{d+1}=1$ then there are no descents in the direction of $e_{d+1}$. This case corresponds to the first term on the right-hand side of \eqref{k1}. Assume that $m_{d+1}>1$ and let $w'\in T(m^{(d)},k)$ such that $v(i,2)=w'(i)$, for every $i\in [m_1]\times\cdots\times [m_d]$. Then $u$, defined by $u(i,j)=v(i,j+1)$ for every $i \in[m_1]\times\cdots\times [m_d]$ and $j\in[m_{d+1}-1]$, is an  $m^{(d+1)}$-tensor word, such that $u(i,1)=w'(i)$, for every $i\in [m_1]\times\cdots\times [m_d]$. Now, for $\ell\geq0$, we have $\des(v)=\ell$ if and only if $\ell = \des(w) + \sum_{i\in[m_1]\times\cdots\times [m_d] } 1_{w(i)>w'(i)} + \des(u)$. This corresponds to the second term on the right-hand side of \eqref{k1} and the proof is complete.
\end{proof}

\begin{theorem}\label{thm;5s}
Define $a_d=\sum_{w\in T(m^{(d)},k)}\des(w)$ and let $A_d(x)$ be the corresponding generating function. Then
\begin{equation}\label{eq;71}
a_d=\frac{1}{2}\left(d\prod_{\ell\in[d]}m_{\ell}-\sum_{i\in[d]}\prod_{\ell\in[d],\ell\neq i}m_{\ell}\right)(k-1)k^{\prod_{\ell\in[d]}m_{\ell}-1}
\end{equation}
and
\begin{equation}\label{eq;72}
A_{d+1}(x)=\frac{a_{d}x+\frac{1}{2}(k-1)k^{2\prod_{\ell\in[d]}m_{\ell}-1}\prod_{\ell\in[d]}m_{\ell}x^{2}}{(1-k^{\prod_{\ell\in[d]}m_{\ell}}x)^{2}}.\end{equation}
\end{theorem}

\begin{proof}
We proceed by induction on $d$. The case $d=1$ is similar to the general case and the details are omitted. 
Assume that \eqref{eq;71} holds for $d$. In order to prove that it holds for $d+1$, we first prove \eqref{eq;72}. To this end, let $w\in T(m^{(d)},k)$. Differentiating \eqref{k1} with respect to $t$ and substituting $t=1$, we obtain
\begin{align}
&A_{d+1}(x,w)\nonumber\\&=\des(w)x+x\sum_{w'\in T(m^{(d)},k)}\left(A_{d+1}(x,w')+\left(\des(w)+\sum_{i\in[m_1]\times\cdots\times [m_d] }1_{w(i)>w'(i)}\right)F(1,x,w')\right)\nonumber\\&=\des(w)x+xA_{d+1}(x)+\left(k^{\prod_{\ell\in[d]} m_\ell}\des(w)+\sum_{w'\in T(m^{(d)},k)}\sum_{i\in[m_1]\times\cdots\times [m_d] }1_{w(i)>w'(i)}\right)\frac{x^2}{1-k^{\prod_{\ell\in[d]} m_\ell}x}.
\label{k2}\end{align}
Summing \eqref{k2} over $w\in T(m^{(d)},k)$ and solving for $A_{d+1}(x)$, we obtain
\begin{align}
A_{d+1}(x)&= \frac{a_dx}{1-k^{\prod_{\ell\in[d]} m_\ell}x}+\left(a_d k^{\prod_{\ell\in[d]} m_\ell} +\sum_{w,w'\in T(m^{(d)},k)}\sum_{i\in[m_1]\times\cdots\times [m_d] }1_{w(i)>w'(i)}\right)\frac{x^2}{(1-k^{\prod_{\ell\in[d]} m_\ell}x)^2}.\nonumber
\end{align}
Now, due to symmetry,
\begin{align}
\sum_{w,w'\in T(m^{(d)},k)}\sum_{i\in[m_1]\times\cdots\times [m_d] }1_{w(i)>w'(i)}
&=\prod_{\ell\in[d]} m_\ell\sum_{w'\in T(m^{(d)},k)}\sum_{w\in T(m^{(d)},k)}1_{w(1,\ldots,1)>w'(1,\ldots,1)}\nonumber\\
&=\prod_{\ell\in[d]} m_\ell k^{\prod_{\ell\in[d]} m_\ell-1}\sum_{w'\in T(m^{(d)},k)}(k-w'(1,\ldots,1))\nonumber  \\
&=\prod_{\ell\in[d]} m_\ell k^{2\prod_{\ell\in[d]} m_\ell-2}\sum_{w'(1,\ldots,1)=1}^k(k-w'(1,\ldots,1))\nonumber  \\
&=\frac{1}{2}(k-1)k^{2\prod_{\ell\in[d]} m_\ell-1}\prod_{\ell\in[d]} m_\ell\nonumber.
\end{align}
It follows that
\[A_{d+1}(x)= \frac{a_dx}{1-k^{\prod_{\ell\in[d]} m_\ell}x}+\left(a_d k^{\prod_{\ell\in[d]} m_\ell} +\frac{1}{2}(k-1)k^{2\prod_{\ell\in[d]} m_\ell-1}\prod_{\ell\in[d]} m_\ell\right)\frac{x^2}{(1-k^{\prod_{\ell\in[d]} m_\ell}x)^2},\] from which \eqref{eq;72} immediately follows.

Now we can prove that \eqref{eq;71} holds for $d+1$. Indeed, by \eqref{eq;72}, we have
\begin{align}
&A_{d+1}(x)\nonumber\\&=\frac{a_dx+\frac{1}{2}(k-1)k^{2\prod_{\ell\in[d]}m_{\ell}-1}\prod_{\ell\in[d]}m_{\ell}x^{2}}{(1-k^{\prod_{\ell\in[d]}m_{\ell}}x)^{2}}\nonumber\\
&=\left(a_dx+\frac{1}{2}(k-1)k^{2\prod_{\ell\in[d]}m_{\ell}-1}\prod_{\ell\in[d]}m_{\ell}x^{2}\right)\sum_{m_{d+1}\geq1}m_{d+1}k^{(m_{d+1}-1)\prod_{\ell\in[d]}m_{\ell}}x^{m_{d+1}-1}\nonumber\\
&=\sum_{m_{d+1}\geq 1}\frac{1}{2}\left(\left(d\prod_{\ell\in[d]}m_{\ell}-\sum_{i\in[d]}\prod_{\ell\in[d],\ell\neq i}m_{\ell}\right)m_{d+1}+\prod_{\ell\in[d]}m_{\ell}(m_{d+1}-1)\right)(k-1)k^{\prod_{\ell\in[d+1]}m_{\ell}-1}x^{m_{d+1}}\nonumber\\
&=\sum_{m_{d+1}\geq1}\frac{1}{2}\left((d+1)\prod_{\ell\in[d+1]}m_{\ell}-\sum_{i\in[d+1]}\prod_{\ell\in[d+1],\ell\neq i}m_{\ell}\right)(k-1)k^{\prod_{\ell\in[d+1]}m_{\ell}-1}x^{m_{d+1}}.\nonumber
\end{align}
Thus, $$a_{d+1}=\frac{1}{2}\left((d+1)\prod_{\ell\in[d+1]}m_{\ell}-\sum_{i\in[d+1]}\prod_{\ell\in[d+1],\ell\neq i}m_{\ell}\right)(k-1)k^{\prod_{\ell\in[d+1]}m_{\ell}-1}$$ and the proof is complete.
\end{proof}


\begin{remark}\label{rem;1}
For the proof of the result regarding the total number of levels, only a few minor modifications are necessary. 
First, equation \eqref{k1} needs to be replaced with
\[
F_{d+1}(t,x,w)=t^{\lev(w)}x+x\sum_{w'\in T(m^{(d)},k)}t^{\lev(w)+\sum_{i\in[m_1]\times\cdots\times [m_d] } 1_{w(i)=w'(i)}}F_{d+!}(t,x,w').
\]
Second, it is not hard to see that
\[
\sum_{w,w'\in T(m^{(d)},k)}\sum_{i\in[m_1]\times\cdots\times [m_d] }1_{w(i)=w'(i)}=k^{2\prod_{\ell\in[d]} m_\ell-1}\prod_{\ell\in[d]} m_\ell.
\]
\end{remark}

\subsection{Cyclic tensor words}

Let $m_{d+1}\in\N$ and let $\ell$ be a nonnegative integer. Set $m^{(d+1)}=(m_1,\ldots,m_{d+1})$. We denote by $D_{m^{(d)},m_{d+1},\ell}$ the number of $m^{(d+1)}$-tensor words $w$ such that $\cycdes(w)=\ell$. For $w_1,w_2\in T(m^{(d)},k)$, we denote by $D_{m^{(d)},m_{d+1},\ell}(w_1,w_2)$ the number of $m^{(d+1)}$-tensor words $w'$, such that $\cycdes(w')=\ell$ and such that $w'(i,r)=w_r(i)$, for every $i\in [m_1]\times\cdots\times [m_d]$ and $r=1,2$. Let $F_{d+1}(t, x)=\sum_{m_{d+1}\geq 1}\sum_{\ell\geq 0} D_{m^{(d)},m_{d+1},\ell}t^\ell x^{m_{d+1}}$ be the generating function of the $D_{m^{(d)},m_{d+1},\ell}$s, and, for $w_1,w_2\in T(m^{(d)},k)$, we denote by $F_{d+1}(t,x,w_1,w_2)=\sum_{m_{d+1}\geq 1}\sum_{\ell\geq 0}D_{m^{(d)},m_{d+1},\ell}(w)t^\ell x^{m_{d+1}}$ the generating function of the $D_{m^{(d)},m_{d+1},\ell}(w_1,w_2)$s. Notice that $$F_{d+1}(1,x,w_1,w_2)=\frac{x^2}{1-k^{\prod_{\ell\in[d]} m_\ell}x}.$$


The following lemma and theorem are the analogues of Lemma \ref{lem;j1} and Theorem \ref{thm;5s}, respectively. We state them without proof. 

\begin{lemma}\label{lem;j100}
For $w_1,w_2\in T(m^{(d)},k)$, the generating function $F(t,x,w_1,w_2)$ satisfies the equation
\begin{align}
&F_{d+1}(t,x,w_1,w_2)=t^{\cycdes(w_1w_2)}x^{2}+\nonumber\\&x\sum_{w'\in T(m^{(d)},k)}t^{\sum_{i\in[m_{1}]\times\cdots\times[m_{d}]}\left(1_{w_1(i)>w_2(i)}+1_{w_2(i)>w'(i)}-1_{w_1(i)>w'(i)}\right)+\cycdes(w')}F_{d+1}(t,x,w_1,w'),\nonumber
\end{align} where $w_1w_2$ is the tensor word $w'\in T(m^{(d+1)},k)$, with $m_{d+1}=2$ and $w'(i,r)=w_r$, for every $i\in [m_1]\times\cdots\times [m_d]$ and $r=1,2$.
\end{lemma}

\begin{theorem}
Define $g_d=\sum_{w\in T(m^{(d)},k)}\cycdes(w)$ and let $G_d(x)$ be the corresponding generating function. Then
\[
g_d=\frac{1}{2}\left(d\prod_{\ell\in[d]}m_{\ell}\right)(k-1)k^{\prod_{\ell\in[d]}m_{\ell}-1}
\]
and
\begin{align}
G_{d+1}(x)&=\frac{2\left(\frac{1}{2}(k-1)k^{2\prod_{\ell\in[d]}m_{\ell}-1}\prod_{\ell\in[d]}m_{\ell}+k^{\prod_{\ell\in[d]}m_{\ell}}g_{d}\right)}{(1-k^{\prod_{\ell\in[d]}m_{\ell}}x)^{2}}x^{2}\nonumber\\&-\frac{\left(\frac{1}{2}(k-1)k^{3\prod_{\ell\in[d]}m_{\ell}-1}\prod_{\ell\in[d]}m_{\ell}+k^{2\prod_{\ell\in[d]}m_{\ell}}g_{d}\right)}{(1-k^{\prod_{\ell\in[d]}m_{\ell}}x)^{2}}x^{3}.\nonumber
\end{align}
\end{theorem}

\subsection{The maximal number of descents}

\begin{lemma}\label{lem;d10}
We have $$\max\{\des(x)\;:\;x\in[k]^n\}=\max\{\cycdes(x)\;:\;x\in[k]^n\}=\left\lfloor \frac{n(k-1)}{k} \right\rfloor.$$
\end{lemma}

\begin{proof}
Let $C(n, k)$ denote the set of all binary sequences of length $n$ having at most $k$ consecutive $1$s. For $y\in C(n,k)$ let us denote by $|y|$ the number of $1$s in $y$. We first show that
\begin{equation}\label{eq;108}
\max\{|y|\;:\;y\in C(n,k)\}=\left\lfloor (n+1)k/(k+1) \right\rfloor
\end{equation} (see \seqnum{A182210} in \cite{SL}). Write $n = \sigma(k+1) + \rho$, where $0\leq \rho< k+1$. We then have the sequence $$y=\overbrace{\overbrace{1\cdots 1}^{k\textnormal{ times}}0\cdots \overbrace{1\cdots 1}^{k\textnormal{ times}}0}^{\sigma\textnormal{ times}}\overbrace{1\cdots 1}^{\rho\textnormal{ times}}\in C(n,k),$$ satisfying $$|y|=n-\sigma=n-\left\lfloor \frac{n}{k+1}\right\rfloor=\left\lfloor \frac{(n+1)k}{k+1}\right\rfloor.$$ If $\sigma=0$ then $n\leq k$ and $y$ consists solely of $n$ $1$s, which is obviously maximal. Assume that $\sigma\geq 1$. Then, treating the $0$s as separators between pigeonholes and the $1$s as pigeons, for every $t\in [\sigma]$, we have $$n-(\sigma-t)=(\sigma-t+1)k+\overbrace{(t-1)k+\rho+t}^{\geq 1}.$$ By the pigeonhole principle, in any binary sequence consisting of $\sigma-t$ $0$s and $n-(\sigma-t)$ $1$s, there exists at least one subsequence of $k+1$ consecutive $1$s.

Having proved \eqref{eq;108}, we now show the connection to descents in words. To this end, we construct maps $\varphi\colon [k]^n\to C(n-1,k-1)$ and $\theta\colon C(n-1,k-1)\to [k]^n$ such that $|\varphi(x)| = \des(x)$ and $\des(\theta(y))=|y|$. Let $x=x_1\cdots x_n\in[k]^n$ and define a binary sequence $\varphi(x)=y=y_1\cdots y_{n-1}$ of length $n-1$ as follows: For $i\in[n-1]$ we set $y_i=1_{x_i>x_{i+1}}$. Clearly, $y\in C(n-1,k-1)$ and  $\des(x)=|y|$. Conversely, let $y=y_1\cdots y_{n-1}\in C(n-1,k-1)$ and let $z_i$ (resp.\ $o_i$) be the length of the $i$th sequence of consecutive $0$s (resp.\ $1$s) in $y$, where $i\in[r]$ for some $r\in\N$. We define $\theta(y)=x=v_1\cdots v_r$, where $$v_i=\begin{cases} \overbrace{o_1+1, \ldots, o_1+1}^{z_1+1\textnormal{ times}}, o_1,o_1-1,\ldots, 1&\textnormal{if } i=1;\\\overbrace{o_i+1, \ldots, o_i+1}^{z_i\textnormal{ times}}, o_i,o_i-1,\ldots, 1&\textnormal{otherwise}.\end{cases}$$
First, we notice that $v_1$ is of length $z_1 + o_1 +1$ and $v_i$ is of length $z_i+o_i$, for $2\leq i\leq r$. Thus, $x$ is of length $n$, since $\sum_{i\in[r]}(z_i+o_i)=n-1$. Second. $0\leq o_i\leq k-1$, for every $i\in[r]$. Thus, $x\in[k]^n$. Finally, since $z_i>0$ for every $2\leq i\leq r$, the number of descents occurring in $v_i$ is $o_i$. Notice that every $v_i$  ends with $1$, so no descents occur in the transition between $v_i$ and $v_{i+1}$. It follows that $\des(x)=\sum_{i\in[r]}o_i=|y|$.
We have thus proved that $\max\{\des(x)\;:\;x\in[k]^n\}=\left\lfloor n(k-1)/k \right\rfloor$. Trivially, $\max\{\des(x)\;:\;x\in[k]^n\}\leq \max\{\cycdes(x)\;:\;x\in[k]^n\}$. On the other hand, let $x\in[k]^n$. Necessarily, there exists $i\in[n]$ such that $x_i\leq x_j$, where $j=i+1$ if $i<n$ and $j=1$, otherwise. Let $x'$ be the rotation of $x$ such the first letter of $x'$ is $x_j$. Then $\cycdes(x)=\cycdes(x')=\des(x')$. Conclude that $\max\{\cycdes(x)\;:\;x\in[k]^n\}\leq \max\{\des(x)\;:\;x\in[k]^n\}$ and the proof is complete.
\end{proof}

\begin{corollary}
Let $w\in T(m^{(d)},k)$. Then
\begin{equation}\label{eq;5w}
\cycdes(w)\leq\sum_{i\in[d]}\left(\prod_{\ell\in[d],\ell\neq i}m_{\ell}\right)\left\lfloor \frac{m_{i}(k-1)}{k}\right\rfloor.
\end{equation}
\end{corollary}

\begin{proof}
We proceed by induction on $d$. The case $d=1$ follows from Lemma \ref{lem;d10}. Assume that \eqref{eq;5w} holds for $d$ and let $w\in T(m^{(d+1)},k)$. Then
\begin{align}
\cycdes(w)&\leq m_{d+1}\max\{\cycdes(w')\;:\;w'\in T(m^{(d)},k)\}+\left(\prod_{\ell\in[d]}m_\ell\right)\left\lfloor \frac{m_{d+1}(k-1)}{k}\right\rfloor\nonumber\\
&\leq m_{d+1}\sum_{i\in[d]}\left(\prod_{\ell\in[d],\ell\neq i}m_{\ell}\right)\left\lfloor \frac{m_{i}(k-1)}{k}\right\rfloor+\left(\prod_{\ell\in[d]}m_\ell\right)\left\lfloor \frac{m_{d+1}(k-1)}{k}\right\rfloor\nonumber\\
&=\sum_{i\in[d+1]}\left(\prod_{\ell\in[d+1],\ell\neq i}m_{\ell}\right)\left\lfloor \frac{m_{i}(k-1)}{k}\right\rfloor, \nonumber
\end{align} as required.
\end{proof}

\begin{example}
Let $w\in T(m^{(2)},k)$. Then $$\cycdes(w)\leq m_1\left\lfloor \frac{m_2(k-1)}{k} \right\rfloor+m_2\left\lfloor \frac{m_1(k-1)}{k} \right\rfloor.$$
\end{example}

\begin{example}
We have
$$\max\{\cycdes(w)\;:\;w\in T((3,4),3)\}=14.$$
Indeed,
\[\cycdes\begin{pmatrix}
1&1&3&2\\
3&3&2&1\\
2&2&1&3
\end{pmatrix}=14.\] On the other hand, \[3\left\lfloor \frac{4\cdot 2}{3} \right\rfloor+4\left\lfloor \frac{3\cdot 2}{3} \right\rfloor=14. \]

\end{example}

\end{document}